\documentclass[a4, 12pt]{amsart}
\usepackage[mathscr]{eucal}
\usepackage{amssymb}
\usepackage{latexsym}
\usepackage{amsthm}
\usepackage{color}
\theoremstyle{plain}
\newtheorem{theorem}{Theorem}[section]

\newtheorem{remark}{Remark}[section]
\newtheorem{lemma}{Lemma}[section]
\newtheorem{proposition}{Proposition}[section]

\makeatletter
\@addtoreset{equation}{section}

\title[The rigidity theorem for complete Lagrangian self-shrinkers]
{The rigidity theorem for complete Lagrangian self-shrinkers}

\author [Z. Li, R. X. Wang and G. Wei]{Zhi Li, Ruixin Wang and Guoxin Wei}

\address{Zhi Li \\  \newline \indent College of Mathematics and Information Science, Henan Normal University,
\newline \indent 453007, Xinxiang, Henan, China.}
\email{lizhihnsd@126.com}

\address{Ruixin Wang\\  \newline \indent College of Mathematics and Information Science, Henan Normal University,
\newline \indent 453007, Xinxiang, Henan, China.}
\email{wangruixin@htu.edu.cn}

\address{Guoxin Wei \\ \newline \indent School of Mathematical Sciences, South China Normal University,
\newline \indent 510631, Guangzhou,  China.}
\email{weiguoxin@tsinghua.org.cn}

\begin{document}
\maketitle

\begin{abstract}
In this paper, we obtain a rigidity result of $2$-dimensional complete lagrangian self-shrinkers with constant squared norm $|\vec{H}|^{2}$ of the mean curvature vector in the Euclidean space $\mathbb{R}^{4}$. The same idea is also used to give a similar result of Lagrangian $\xi$-submanifolds in $\mathbb{R}^{4}$.
\end{abstract}


\section{introduction}
\vskip2mm
\noindent

An immersion $x: M^{n}\to \mathbb{R}^{n+p}$ of a smooth manifold $M^{n}$ into Euclidean space is said to be a self-shrinker if it satisfies the quasilinear elliptic system
\begin{align}\label{1.1-1}
\vec{H}+x^{\perp}=0,
\end{align}
where $\vec{H}$ is the mean curvature vector and $x^{\perp}$ is the orthogonal projection of the position vector $x$ in $\mathbb{R}^{n+p}$ to the normal bundle of $M^{n}$.
It is well known that the solutions of \eqref{1.1-1} not only give rise to
homothetically shrinking solutions of the mean curvature flow but also they
play an interesting role in the formation of type-1 singularities(\cite{Hui90}). Besides,
self-shrinkers may equivalently be interpreted as critical points for the Gaussian area $\int_{M}e^{-|x|^{2}/2}d\sigma$,
hence also as minimal hypersurfaces for the conformal metric $g_{ij}=e^{-|x|^{2}/n}\delta_{ij}$.
In fact, classification and rigidity of self-shrinkers have been studied extensively for  hypersurfaces or arbitrary codimension. To learn more about them, please refer to (\cite{CL13}, \cite{CP15},  \cite{CW15}, \cite{CZ13}, \cite{CW}, \cite{DX13}, \cite{DX14}, \cite{Hui90}, \cite{LN11}, \cite{LW14}, \cite{Smo05}, \cite{XX17}) and the references therein.

It is known that the mean curvature flow keep invariant of the Lagrangian property, which means that if the initial submanifold $x: M^{n}\to \mathbb{R}^{2n}$ is Lagrangian, then the mean curvature flow $x(\cdot, t): M^{n}\to \mathbb{R}^{2n}$ is also Lagrangian.
From view points of submanifolds theory, it is also very natural to study rigidity
and classification theorems for the Lagrangian self-shrinkers. In (\cite{Cas}), Castro and
Lerma gave a classification of Hamiltonian stationary Lagrangian self-shrinkers in
$\mathbb{R}^{4}$
and in (\cite{Cas1}),  they provided several rigidity results for the Clifford torus in the class of compact self-shrinkers for Lagrangian mean curvature flow. In 2017, Li and Wang (\cite{LW3}) prove a rigidity theorem which improves a previous theorem
by Castro and Lerma (\cite{Cas1}). Later, Cheng, Hori and Wei (\cite{CHW}) established an interesting classification theorem for complete Lagrangian self-shrinkers with
constant squared norm of the second fundamental form in $\mathbb{R}^{4}$. To see the details for Lagrangian self-shrinkers, readers are referred to (\cite{Anc06}, \cite{LW10},  \cite{Neves11}, \cite{Smo00}).

Specifically, Castro and Lerma (\cite{Cas1}) proved that Clifford torus $S^{1}(1)\times S^{1}(1)$ is the only compact Lagrangian self-shrinker with constant squared norm $|\vec{H}|^{2}$ of the mean curvature vector in $\mathbb{R}^{4}$.
It is natural to ask the following problems:
To classify $2$-dimensional complete Lagrangian self-shrinkers in $\mathbb{R}^{4}$
if the squared norm $|\vec{H}|^{2}$ of the mean curvature vector is constant.
For the above problem, we solve it under the assumption of the squared norm of the second fundamental form being bounded from above.  In fact, we prove the following:

\begin{theorem}\label{theorem 1.1}
Let $x: M^{2}\to \mathbb{R}^{4}$ be a
$2$-dimensional complete Lagrangian self-shrinker with constant squared norm $|\vec{H}|^{2}$ of the mean curvature vector in $\mathbb{R}^{4}$. If the squared norm of the second fundamental form is bounded from above, then $x(M^{2})$ is isometric to one of
$\mathbb {R}^{2}$, $S^{1}(1)\times\mathbb {R}^{1}$ and $S^{1}(1)\times S^{1}(1)$.
\end{theorem}

It is known that the concept of self-shrinker
of the mean curvature flow can be naturally generalized to the $\xi$-submanifold. Namely,
$M^{n}$ is called a $\xi$-submanifold if $\xi=\vec{H}+x^{\perp}$
is parallel in the normal bundle. It is easy to see that when $\xi=0$, $M^{n}$ is a self-shrinker. Thus, by an application of Theorem \ref{theorem 1.1} and a theorem
of Hoffman (Theorem 4.1, \cite{Hoff}), we can easily obtain the following more general result.

\begin{theorem}\label{theorem 1.2}
Let $x: M^{2}\to \mathbb{R}^{4}$ be a $2$-dimensional complete Lagrangian $\xi$-submanifold with constant squared norm $|\vec{H}|^{2}$ of the mean curvature vector in $\mathbb{R}^{4}$. If the squared norm of the second fundamental form is bounded from above, then $x(M^{2})$ must be isometric to one of the following canonical submanifolds:
\begin{enumerate}
\item a $2$-plane not necessarily passing through the origin;
\item a a circular cylinder $S^{1}(r)\times\mathbb {R}^{1}$ for some $r>0$;
\item  a Clifford torus $S^{1}(r_{1})\times S^{1}(r_{2})$ for some $r_{1}, \ r_{2}>0$
\end{enumerate}
\end{theorem}
\begin{remark}
Note that there does not exist any
Lagrangian $\xi$-submanifold in $\mathbb{R}^{4}$ with the topology of a sphere (see \cite{LC16}).
\end{remark}

\vskip5mm
\section {Preliminaries}
\vskip2mm

\noindent

Let $x: M^{2} \rightarrow \mathbb{R}^{4}$ be an
$2$-dimensional connected submanifold of $4$-dimensional
Euclidean space $\mathbb{R}^{4}$ with $\mathbb{C}^{2}$. Denote
by $J$ the canonical complex structure on $\mathbb C^{2}$.
We choose orthonormal tangent vector fields $\{e_{1}, e_{2}\}$ and $\{e_{1^{\ast}}, e_{2^{\ast}}\}$ are normal vector fields given by
$$e_{1^{\ast}}=J e_{1}, \ e_{2^{\ast}}=Je_{2}.$$
Then
$$\{e_{1}, e_{2}, e_{1^{\ast}}, e_{2^{\ast}}\}$$
is called an adapted Lagrangian frame field. The dual frame fields of $\{e_{1}, e_{2}\}$ are $\{\omega_{1}, \omega_{2}\}$, the Levi-Civita connection forms and normal
connection forms are $\omega_{ij}$ and $\omega_{i^{\ast}j^{\ast}}$ , respectively.

The second fundamental form $h$ and the mean curvature $\vec{H}$ of $x$ are respectively
defined by $$h=\sum_{ijp}h^{p^{\ast}}_{ij}\omega_{i}\otimes\omega_{j}\otimes e_{p^{\ast}},\ \ \vec{H}=\sum_{p}H^{p^{\ast}}e_{p^{\ast}}=\sum_{i,p}h^{p^{\ast}}_{ii}e_{p^{\ast}}.$$

Let $S=\sum_{i,j,p}(h^{p^{\ast}}_{ij})^{2}$ be the squared norm of the second
fundamental form and $H=|\vec{H}|$ denote the mean curvature of $x$. If we denote the components of curvature tensors of the Levi-Civita connection forms $\omega_{ij}$ and normal connection forms $\omega_{i^{\ast}j^{\ast}}$ by
$R_{ijkl}$ and $R_{i^{\ast}j^{\ast}kl}$, respectively, then the equations of Gauss, Codazzi and Ricci are given by
\begin{align}\label{2.1-1}
R_{ijkl}=\sum_{p}(h^{p^{\ast}}_{ik}h^{p^{\ast}}_{jl}-h^{p^{\ast}}_{il}h^{p^{\ast}}_{jk}),
\end{align}
\begin{align}\label{2.1-2}
R_{ik}=\sum_{p}H^{p^{\ast}}h^{p^{\ast}}_{ik}-\sum_{j,p}h^{p^{\ast}}_{ij}h^{p^{\ast}}_{jk},
\end{align}
\begin{align}\label{2.1-3}
h^{p^{\ast}}_{ijk}=h^{p^{\ast}}_{ikj},
\end{align}
\begin{align}\label{2.1-4}
R_{p^{\ast}q^{\ast}kl}=\sum_{i}(h^{p^{\ast}}_{ik}h^{q^{\ast}}_{il}
-h^{p^{\ast}}_{il}h^{p^{\ast}}_{ik}),
\end{align}
\begin{align}\label{2.1-5}
R=H^{2}-S.
\end{align}

Since $x: M^{2} \rightarrow\mathbb R^{4}$ is a Lagrangian surface (\cite{LW2}), we have
\begin{align}\label{2.1-6}
h^{p^{\ast}}_{ij}=h^{p^{\ast}}_{ji}=h^{i^{\ast}}_{pj}, \ \ i,j,p=1,2.
\end{align}
From \eqref{2.1-3} and \eqref{2.1-6}, we easily know that the components $h^{p^{\ast}}_{ijk}$ is totally symmetric for $i, j, k, p$. In particular,
\begin{align}\label{2.1-7}
h^{p^{\ast}}_{ijk}=h^{p^{\ast}}_{kji}=h^{i^{\ast}}_{pjk}, \ \ i, j, k ,p=1, 2.
\end{align}
By use of \eqref{2.1-1}, \eqref{2.1-2} and \eqref{2.1-5}, we obtain
\begin{align}\label{2.1-8}
R_{ijkl}=K(\delta_{ik}\delta_{jl}-\delta_{il}\delta_{jk})=R_{i^{\ast}j^{\ast}kl}, \ \ K=\frac{1}{2}(H^{2}-S),
\end{align}
where $K$ is the Gaussian curvature of $x$.

\noindent By defining
\begin{align*}
\sum_{l}h^{p^{\ast}}_{ijkl}\omega_{l}=dh^{p^{\ast}}_{ijk}+\sum_{l}h^{p^{\ast}}_{ljk}\omega_{li}
+\sum_{l}h^{p^{\ast}}_{ilk}\omega_{lj}+\sum_{l} h^{p^{\ast}}_{ijl}\omega_{lk}+\sum_{q} h^{q^{\ast}}_{ijk}\omega_{q^{\ast}p^{\ast}},
\end{align*}
we have the following Ricci identity
\begin{align}\label{2.1-9}
h^{p^{\ast}}_{ijkl}-h^{p^{\ast}}_{ijlk}=\sum_{m}
h^{p^{\ast}}_{mj}R_{mikl}+\sum_{m} h^{p^{\ast}}_{im}R_{mjkl}+\sum_{m} h^{m^{\ast}}_{ij}R_{m^{\ast}p^{\ast}kl}.
\end{align}

Define a differential operator
\begin{equation*}
\mathcal{L}f=\Delta f-\langle x,\nabla f\rangle,
\end{equation*}
where $\Delta$ and $\nabla$ denote the Laplacian and the gradient
operator, respectively. The following generalized maximum principle
of Omori-Yau type concerning the operator $\mathcal{L}$ will
be used in this paper, which was proved by Cheng and Peng \cite{CP15}.

\begin{lemma}\label{lemma 2.1}

Let $x:M^{n}\to \mathbb{R}^{n+p}$
$(p\geq1)$ be a complete self-shrinker with Ricci curvature bounded from below. Let $f$
be any $C^{2}$-function bounded from above on this self-shrinker. Then, there exists a
sequence of points $\{p_{t}\} \subset M^{n}$, such that
\begin{equation*}
\lim_{t\rightarrow\infty} f(p_{t})=\sup f,\quad
\lim_{t\rightarrow\infty} |\nabla f|(p_{t})=0,\quad
\lim_{t\rightarrow\infty}\mathcal{L}f(p_{t})\leq 0.
\end{equation*}
\end{lemma}

For the mean curvature vector field $\vec{H}=\sum_{p}H^{p^{\ast}}e_{p^{\ast}}$, we define
\begin{align}\label{2.1-10}
|\nabla^{\perp}\vec{H}|^{2}=\sum_{i,p}(H^{p^{\ast}}_{,i})^{2}, \ \ \Delta^{\perp}H^{p^{\ast}}=\sum_{i}H^{p^{\ast}}_{,ii}.
\end{align}
By the definition \eqref{1.1-1} of lagrangian self-shrinker, it is sufficient
to give several basic differential formulas.

\begin{align}\label{2.1-11}
H^{p^{\ast}}_{,i}
=\sum_{k}h^{p^{\ast}}_{ik}\langle x, e_{k}\rangle, \ \
H^{p^{\ast}}_{,ij}
=\sum_{k}h^{p^{\ast}}_{ijk}\langle x, e_{k}\rangle+h^{p^{\ast}}_{ij}-\sum_{k,q}h^{p^{\ast}}_{ik}h^{q^{\ast}}_{kj}H^{q^{\ast}}.
\end{align}
Using the above formulas and the Ricci identities, we can get the following Lemmas (see \cite{CHW}).

\begin{lemma}\label{lemma 2.2}
Let $x:M^{2}\rightarrow \mathbb{R}^{4}$ be a $2$-dimensional complete lagrangian self-shrinker. We have
\begin{align}\label{2.1-12}
\frac{1}{2}\mathcal{L} H^{2}=\sum_{i,p}(H^{p^{\ast}}_{,i})^{2}+H^{2}-\sum_{i,j}(\sum_{p}H^{p^{\ast}}h^{p^{\ast}}_{ij})^{2}
\end{align}
and
\begin{align}\label{2.1-13}
\frac{1}{2}\mathcal{L}S
=\sum_{i,j,k,p}(h^{p^{\ast}}_{ijk})^{2}+S(1-\frac{3}{2}S)+2H^{2}S-\frac{1}{2}H^{4}
-\sum_{i,j}(\sum_{p}H^{p^{\ast}}h^{p^{\ast}}_{ij})^{2}.
\end{align}
\end{lemma}

\begin{lemma}\label{lemma 2.3}
Let $x:M^{2}\rightarrow \mathbb{R}^{4}$ be a $2$-dimensional complete lagrangian self-shrinker. If $|\vec{H}|^{2}=constant$, we have
\begin{align}\label{2.1-14}
&\frac{1}{2}\mathcal{L}\sum_{i,p}(H^{p^{\ast}}_{,i})^{2}\\
=&\sum_{i,j,p}(H^{p^{\ast}}_{,ij})^{2}+(3K+2)\sum_{i,p}(H^{p^{\ast}}_{,i})^{2} \nonumber
-2\sum_{i,j,k,p,q}H^{p^{\ast}}_{,i}h^{p^{\ast}}_{ijk}h^{q^{\ast}}_{jk}H^{q^{\ast}} \\
&-\sum_{i,j,k,p,q}H^{p^{\ast}}_{,i}h^{p^{\ast}}_{il}h^{q^{\ast}}_{kl}H^{q^{\ast}}_{,k} \nonumber
\end{align}
and
\begin{align}\label{2.1-15}
&\frac{1}{2}\mathcal{L}\sum_{i,j}(\sum_{p}H^{p^{\ast}}h^{p^{\ast}}_{ij})^{2}\\
=&\nonumber(3K+2)\sum_{j,k,p,q}H^{p^{\ast}}h^{p^{\ast}}_{jk}H^{q^{\ast}}h^{q^{\ast}}_{jk}
-K\big(H^{4}+\sum_{j,k,p}h^{p^{\ast}}_{jk}H^{p^{\ast}}H^{j^{\ast}}H^{k^{\ast}} \big) \\
&\nonumber-\sum_{i,j,k,l,p,q}H^{p^{\ast}}h^{p^{\ast}}_{ij}h^{q^{\ast}}_{ij}h^{q^{\ast}}_{kl}h^{m^{\ast}}_{kl}H^{m^{\ast}}
-\sum_{j,k,l,m,p,q}H^{p^{\ast}}h^{p^{\ast}}_{jk}H^{q^{\ast}}h^{q^{\ast}}_{jl}H^{m^{\ast}}h^{m^{\ast}}_{kl}\\
&\nonumber+2\sum_{i,j,k,p,q}H^{p^{\ast}}_{,i}h^{p^{\ast}}_{ijk}h^{q^{\ast}}_{jk}H^{q^{\ast}}
+\sum_{i,j,k}\big(\sum_{p}(H^{p^{\ast}}_{,i}h^{p^{\ast}}_{jk}+H^{p^{\ast}}h^{p^{\ast}}_{ijk})\big)
\big(\sum_{q}(H^{q^{\ast}}_{,i}h^{q^{\ast}}_{jk} \\
&\nonumber+H^{q^{\ast}}h^{q^{\ast}}_{ijk})\big).
\end{align}
Thus, using \eqref{2.1-12}, \eqref{2.1-14} and \eqref{2.1-15}, we infer that
\begin{align}\label{2.1-16}
&\sum_{i,j,p}(H^{p^{\ast}}_{,ij})^{2} \\
=&\nonumber (3K+2)\sum_{j,k,p,q}H^{p^{\ast}}h^{p^{\ast}}_{jk}H^{q^{\ast}}h^{q^{\ast}}_{jk}
-K\big(H^{4}+\sum_{j,k,p}h^{p^{\ast}}_{jk}H^{p^{\ast}}H^{j^{\ast}}H^{k^{\ast}} \big) \\
&\nonumber-\sum_{i,j,k,l,p,q}H^{p^{\ast}}h^{p^{\ast}}_{ij}h^{q^{\ast}}_{ij}h^{q^{\ast}}_{kl}h^{m^{\ast}}_{kl}H^{m^{\ast}}
-\sum_{j,k,l,m,p,q}H^{p^{\ast}}h^{p^{\ast}}_{jk}H^{q^{\ast}}h^{q^{\ast}}_{jl}H^{m^{\ast}}h^{m^{\ast}}_{kl} \\
&\nonumber+4\sum_{i,j,k,p,q}H^{p^{\ast}}_{,i}h^{p^{\ast}}_{ijk}h^{q^{\ast}}_{jk}H^{q^{\ast}}-(3K+2)\sum_{i,p}(H^{p^{\ast}}_{,i})^{2}
+\sum_{i,j,k,p,q}H^{p^{\ast}}_{,i}h^{p^{\ast}}_{il}h^{q^{\ast}}_{kl}H^{q^{\ast}}_{,k} \\
&\nonumber+\sum_{i,j,k}\big(\sum_{p}(H^{p^{\ast}}_{,i}h^{p^{\ast}}_{jk}+H^{p^{\ast}}h^{p^{\ast}}_{ijk})\big)
\big(\sum_{q}(H^{q^{\ast}}_{,i}h^{q^{\ast}}_{jk}+H^{q^{\ast}}h^{q^{\ast}}_{ijk})\big).
\end{align}

\end{lemma}

 \vskip10mm
\section{Proof of Theorem 1.1}

\vskip2mm
To draw the conclusion of Theorem \ref{theorem 1.1}, we need the following proposition.

\begin{proposition}\label{proposition 3.1}
Let $x:M^{2}\rightarrow \mathbb{R}^{4}$
be a $2$-dimensional complete Lagrangian self-shrinker with non-zero constant squared norm $|\vec{H}|^{2}$ of the mean curvature vector. If the squared norm $S$ of the second fundamental form is bounded from above, then $H^{2}=2$ and $\sup S=2$, or $H^{2}=1$ and $\sup S=1$.
\end{proposition}
\begin{proof}

Since $|\vec{H}|^{2}\neq0$, we choose a local frame field $\{e_{1}, e_{2}\}$
such that
$$ \vec{H}=H^{1^{\ast}}e_{1^{\ast}}, \ \ H^{1^{\ast}}=|\vec{H}|=H, \ \ H^{2^{\ast}}=h^{2^{\ast}}_{11}+h^{2^{\ast}}_{22}=0.
$$
Then,
$$S=(h^{1^{\ast}}_{11})^{2}+3(h^{1^{\ast}}_{22})^{2}+4(h^{2^{\ast}}_{11})^{2}, \ \ H^{2}=(h^{1^{\ast}}_{11}+h^{1^{\ast}}_{22})^{2}\leq \frac{4}{3}\big((h^{1^{\ast}}_{11})^{2}+3(h^{1^{\ast}}_{22})^{2}\big)\leq \frac{4}{3}S$$
and the equality of the above inequality holds if and only if
$$h^{1^{\ast}}_{11}=3h^{1^{\ast}}_{22}, \ \ h^{2^{\ast}}_{11}=0.$$
Since $S$ is bounded from above, from the Gauss equations, we know that the Ricci curvature of $x:M^{2}\rightarrow \mathbb{R}^{4}$ is bounded from below. By applying the generalized maximum principle of Omori-Yau type concerning the operator $\mathcal{L}$ to the function $S$, there exists a sequence $\{p_{m}\} \in M^{2}$ such that
\begin{align*}
\lim_{t\rightarrow\infty} S(p_{m})=\sup S, \ \
\lim_{t\rightarrow\infty} |\nabla S|(p_{m})=0, \ \
\lim_{t\rightarrow\infty}\mathcal{L} S(p_{m})\leq0.
\end{align*}
Since $|\vec{H}|^{2}$ is constant, \eqref{2.1-12} implies
\begin{align}\label{3.1-1}
\sum_{i,p}(H^{p^{\ast}}_{,i})^{2}
=\sum_{i,j}(\sum_{p}H^{p^{\ast}}h^{p^{\ast}}_{ij})^{2}-H^{2}.
\end{align}
For $S$ being bounded from above, we know that
$\{h^{p^{\ast}}_{ij}(p_{m})\}$ are bounded sequences for $ i, j, p = 1,2$.
Hence we can assume
\begin{align*}
&\lim_{t\rightarrow\infty} S(p_{m})=\sup S=\bar S, \ \ \lim_{t\rightarrow\infty}K(p_{m})=\bar K, \ \ \lim_{t\rightarrow\infty}h^{p^{\ast}}_{ij}(p_{m})=\bar h^{p^{\ast}}_{ij}, \\ &\lim_{t\rightarrow\infty}H^{p^{\ast}}_{,i}(p_{m})=\bar H^{p^{\ast}}_{,i},
 \ \ i, j, p=1, 2.
\end{align*}

If $\sup S=0$, naturally obtained that $S\equiv0$ and $|\vec{H}|^{2}\equiv0$. This contradicts the assumption of the Proposition \ref{proposition 3.1}. The following text will only consider $\sup S>0$, then there exists a subsequence $\{p_{t}\}\subset\{p_{m}\}$ such that $S(p_{t})\neq0$, which implies $h^{p^{\ast}}_{ij}(p_{t})\neq0$ for some $i, j, p=1, 2$.
Unless otherwise specified, the following equations are considered at point $p_{t}\in M^{2}$.

\noindent
Since $|\nabla H^{2}|=0$ and $|\nabla H^{2}|^{2}=4\sum_{k}(\sum_{p}H^{p^{\ast}}H^{p^{\ast}}_{,k})^{2}$,
we can see that
\begin{align}\label{3.1-2}
H^{1^{\ast}}_{,k}=H^{k^{\ast}}_{,1}=0, \ \ h^{1^{\ast}}_{11k}+h^{1^{\ast}}_{22k}=0, \ \ k=1, 2
\end{align}
and
\begin{align}\label{3.1-3}
\sum_{i,p}(H^{p^{\ast}}_{,i})^{2}=(H^{2^{\ast}}_{,2})^{2}=H^{2}\big(\sum_{i,j}(h^{1^{\ast}}_{ij})^{2}-1\big).
\end{align}
from \eqref{3.1-1}.

\noindent
Using $H^{2^{\ast}}=h^{2^{\ast}}_{11}+h^{2^{\ast}}_{22}=0$, the first equation of \eqref{2.1-11} can be written as
\begin{align}\label{3.1-4}
H^{1^{\ast}}_{,1}=h^{1^{\ast}}_{11}\langle x, e_{1} \rangle+h^{2^{\ast}}_{11}\langle x, e_{2} \rangle, \ \
H^{1^{\ast}}_{,2}=h^{2^{\ast}}_{11}\langle x, e_{1} \rangle+h^{1^{\ast}}_{22}\langle x, e_{2} \rangle
\end{align}
and
\begin{align}\label{3.1-5}
H^{2^{\ast}}_{,1}=h^{2^{\ast}}_{11}\langle x, e_{1} \rangle+h^{1^{\ast}}_{22}\langle x, e_{2} \rangle, \ \
H^{2^{\ast}}_{,2}=h^{1^{\ast}}_{22}\langle x, e_{1} \rangle-h^{2^{\ast}}_{11}\langle x, e_{2} \rangle.
\end{align}
Choosing $\nabla_{k}S=2a_{k}$, \eqref{3.1-2} and $\lim_{t\rightarrow\infty} |\nabla S|(p_{t})=0$ imply that
\begin{align}\label{3.1-6}
(h^{1^{\ast}}_{11}-3h^{1^{\ast}}_{22}) h^{1^{\ast}}_{11k}+3h^{2^{\ast}}_{11}h^{2^{\ast}}_{11k}-h^{2^{\ast}}_{11}h^{2^{\ast}}_{22k}=a_{k}, \ \ \lim_{t\rightarrow\infty}a_{k}(p_{t})=0, \ \ k=1, 2.
\end{align}
Combining \eqref{3.1-2} and \eqref{3.1-6}, we infer

\begin{align}\label{3.1-7}
&\big((h^{1^{\ast}}_{11}-3h^{1^{\ast}}_{22})^{2}+12(h^{2^{\ast}}_{11})^{2}\big)h^{1^{\ast}}_{111}
+4(h^{2^{\ast}}_{11})^{2} h^{2^{\ast}}_{222}=b_{1}, \\
&\nonumber\big((h^{1^{\ast}}_{11}-3h^{1^{\ast}}_{22})^{2}+12(h^{2^{\ast}}_{11})^{2}\big) h^{2^{\ast}}_{111}-h^{2^{\ast}}_{11}(h^{1^{\ast}}_{11}-3h^{1^{\ast}}_{22})h^{2^{\ast}}_{222}
=b_{2}.
\end{align}
where $b_{1}=(h^{1^{\ast}}_{11}-3h^{1^{\ast}}_{22})a_{1}-4h^{2^{\ast}}_{11} a_{2}$,
$b_{2}=3h^{2^{\ast}}_{11} a_{1}+(h^{1^{\ast}}_{11}-3h^{1^{\ast}}_{22})a_{2}$ and
$\lim_{t\rightarrow\infty}b_{k}(p_{t})=0$ for $k=1,2$.

\noindent
From  $|\vec{H}|^{2}=constant$, we naturally obtain
\begin{align}\label{3.1-8}
\sum_{p}H^{p^{\ast}}_{,i}H^{p^{\ast}}_{,j}
+H^{1^{\ast}}H^{1^{\ast}}_{,ij}=0, \ \ i,j=1,2
\end{align}
which implies
\begin{align*}
 H^{1^{\ast}}_{,11}=H^{1^{\ast}}_{,12}=0.
\end{align*}
from \eqref{3.1-2}.
Besides, \eqref{3.1-1} yields that
\begin{align*}
\sum_{i,p}H^{p^{\ast}}_{,i}H^{p^{\ast}}_{,ik}
=\sum_{i,j}(\sum_{p}H^{p^{\ast}}h^{p^{\ast}}_{ij})(\sum_{p}H^{p^{\ast}}_{,k}h^{p^{\ast}}_{ij}
+\sum_{p}H^{p^{\ast}}h^{p^{\ast}}_{ijk}).
\end{align*}
It immediately follows, using \eqref{3.1-2}, we obtain
\begin{align}\label{3.1-9}
H^{2^{\ast}}_{,2}H^{2^{\ast}}_{,21}
=& H^{2}\sum_{i,j}h^{1^{\ast}}_{ij}h^{1^{\ast}}_{ij1}
=H^{2}\big((h^{1^{\ast}}_{11}-h^{1^{\ast}}_{22})h^{1^{\ast}}_{111}+2 h^{2^{\ast}}_{11}h^{2^{\ast}}_{111}\big), \\
H^{2^{\ast}}_{,2}H^{2^{\ast}}_{,22}
=&\nonumber HH^{2^{\ast}}_{,2}\sum_{i,j}h^{1^{\ast}}_{ij}h^{2^{\ast}}_{ij}+
H^{2}\sum_{i,j}h^{1^{\ast}}_{ij}h^{1^{\ast}}_{ij2} \\
=&\nonumber H^{2}\big(h^{2^{\ast}}_{11}H^{2^{\ast}}_{,2} +(h^{1^{\ast}}_{11}-h^{1^{\ast}}_{22})h^{2^{\ast}}_{111}-2 h^{2^{\ast}}_{11}h^{1^{\ast}}_{111}\big) \\
=&\nonumber H^{2}\big((h^{1^{\ast}}_{11}-h^{1^{\ast}}_{22})h^{2^{\ast}}_{111}- h^{2^{\ast}}_{11}(3h^{1^{\ast}}_{111}-h^{2^{\ast}}_{222})\big).
\end{align}
Applying \eqref{3.1-6}, we obtain
\begin{align}\label{3.1-10}
&(h^{1^{\ast}}_{11}-3 h^{1^{\ast}}_{22})h^{1^{\ast}}_{111}+4 h^{2^{\ast}}_{11}h^{2^{\ast}}_{111}=a_{1}, \\
&\nonumber (h^{1^{\ast}}_{11}-3h^{1^{\ast}}_{22}) h^{2^{\ast}}_{111}-h^{2^{\ast}}_{11}(3h^{1^{\ast}}_{111}+h^{2^{\ast}}_{222})=a_{2}.
\end{align}
By \eqref{3.1-9}, \eqref{3.1-10} and $H=h^{1^{\ast}}_{11}+h^{1^{\ast}}_{22}$, we drive that
\begin{align}\label{3.1-11}
&H^{2^{\ast}}_{,2}H^{2^{\ast}}_{,21}
=\frac{1}{2}H^{2}(Hh^{1^{\ast}}_{111}+a_{1}), \\
&\nonumber H^{2^{\ast}}_{,2}H^{2^{\ast}}_{,22}=H^{2}(2h^{1^{\ast}}_{22}h^{2^{\ast}}_{111}
+2h^{2^{\ast}}_{11}h^{2^{\ast}}_{222}+a_{2}).
\end{align}
It follows from the second equation of \eqref{2.1-11} that

\begin{align}\label{3.1-12}
&H^{1^{\ast}}_{,11}=\sum_{k}h^{1^{\ast}}_{11k}\langle x, e_{k}\rangle+h^{1^{\ast}}_{11}-H\sum_{k}(h^{1^{\ast}}_{1k})^{2}, \\
&\nonumber H^{2^{\ast}}_{,21}=\sum_{k}h^{2^{\ast}}_{21k}
\langle x, e_{k}\rangle+h^{1^{\ast}}_{22}-H\sum_{k}h^{1^{\ast}}_{1k}h^{2^{\ast}}_{2k}, \\
&\nonumber H^{1^{\ast}}_{,12}=\sum_{k}h^{1^{\ast}}_{12k}\langle x, e_{k}\rangle+h^{2^{\ast}}_{11}-H\sum_{k}h^{1^{\ast}}_{1k}h^{1^{\ast}}_{2k}, \\
&\nonumber H^{2^{\ast}}_{,22}=\sum_{k}h^{2^{\ast}}_{22k}
\langle x, e_{k}\rangle+h^{2^{\ast}}_{22}-H\sum_{k}h^{1^{\ast}}_{2k}h^{2^{\ast}}_{2k}.
\end{align}
Then by \eqref{2.1-7}, \eqref{3.1-2} and $H^{1^{\ast}}_{,11}=H^{1^{\ast}}_{,12}=0$, we obtain
\begin{align}\label{3.1-13}
H^{2^{\ast}}_{,21}=H\big(1-(h^{1^{\ast}}_{11})^{2}
-h^{1^{\ast}}_{11}h^{1^{\ast}}_{22}\big), \ \
H^{2^{\ast}}_{,22}=H^{2^{\ast}}_{,2}\langle x, e_{2}\rangle-H^{2}h^{2^{\ast}}_{11}.
\end{align}
By use of \eqref{3.1-2}, \eqref{3.1-3}, the first equation of \eqref{3.1-11} and \eqref{3.1-13}, we have that
\begin{align}\label{3.1-14}
&\big(\lim_{t\rightarrow\infty}h^{1^{\ast}}_{111}(p_{t})\big)^{2}
=\big(\lim_{t\rightarrow\infty}h^{1^{\ast}}_{221}(p_{t})\big)^{2}=\frac{4}{H^{2}}\big(\sum_{i,j}(\bar h^{1^{\ast}}_{ij})^{2}-1\big)\big(1-(\bar h^{1^{\ast}}_{11})^{2} \\
&\nonumber -\bar h^{1^{\ast}}_{11}\bar h^{1^{\ast}}_{22}\big)^{2}, \ \
|\lim_{t\rightarrow\infty}h^{2^{\ast}}_{222}(p_{t})|<\infty.
\end{align}
From \eqref{3.1-2} and \eqref{3.1-4}, we get that
\begin{align}\label{3.1-15}
\big(\bar h^{1^{\ast}}_{11}\bar h^{1^{\ast}}_{22}-(\bar h^{2^{\ast}}_{11})^{2}\big)\lim_{t\rightarrow\infty} \langle x, e_{1} \rangle(p_{t})=0, \ \
\big(\bar h^{1^{\ast}}_{11}\bar h^{1^{\ast}}_{22}-(\bar h^{2^{\ast}}_{11})^{2}\big)\lim_{t\rightarrow\infty} \langle x, e_{2} \rangle(p_{t})=0.
\end{align}

In order to obtain our conclusion, our proof subject is divided into the following two cases:

\noindent{\bf Case 1: $\bar h^{1^{\ast}}_{11}\bar h^{1^{\ast}}_{22}
-(\bar h^{2^{\ast}}_{11})^{2}\neq 0$.}

\noindent
Under the condition of Case 1,
\eqref{3.1-15} yields
\begin{align*}
\lim_{t\rightarrow\infty} \langle x, e_{1} \rangle(p_{t})=\lim_{t\rightarrow\infty} \langle x, e_{2} \rangle(p_{t})=0.
\end{align*}
Thus, it follows from \eqref{3.1-3}, \eqref{3.1-5} and \eqref{3.1-14} that
\begin{align}\label{3.1-16}
\bar H^{2^{\ast}}_{,2}=0, \ \ \sum_{i,j}(\bar h^{1^{\ast}}_{ij})^{2}=(\bar h^{1^{\ast}}_{11})^{2}+(\bar h^{1^{\ast}}_{22})^{2}+2(\bar h^{2^{\ast}}_{11})^{2}=1
\end{align}
and
\begin{align}\label{3.1-17}
\lim_{t\rightarrow\infty}h^{1^{\ast}}_{111}(p_{t})
=\lim_{t\rightarrow\infty}h^{1^{\ast}}_{221}(p_{t})=\lim_{t\rightarrow\infty}h^{2^{\ast}}_{222}(p_{t})=0.
\end{align}
At the same time, we naturally know that $$|\lim_{t\rightarrow\infty}h^{2^{\ast}}_{111}(p_{t})|<\infty,$$ otherwise we would use
$|\lim_{t\rightarrow\infty}h^{2^{\ast}}_{111}(p_{t})|=\infty$,
\eqref{3.1-10} and \eqref{3.1-17} yield
\begin{align}\label{3.1-18}
\bar h^{2^{\ast}}_{11}=0, \ \ \bar h^{1^{\ast}}_{11}=3\bar h^{1^{\ast}}_{22}.
\end{align}
By use of the second equation of \eqref{3.1-11}, \eqref{3.1-13} and \eqref{3.1-18}, we obtain
\begin{align*}
\lim_{t\rightarrow\infty}H^{2^{\ast}}_{,22}(p_{t})=0, \ \ \bar h^{1^{\ast}}_{11}=3\bar h^{1^{\ast}}_{22}=0,
\end{align*}
which contradicts the fact that $H\neq0$.

\noindent
Based on the above conclusion, combining $H^{1^{\ast}}_{,11}=H^{1^{\ast}}_{,12}=0$ and the first and third equations of \eqref{3.1-12}, we get that
\begin{align*}
\bar h^{1^{\ast}}_{11}-H\sum_{k}(\bar h^{1^{\ast}}_{1k})^{2}=0, \ \ \bar h^{2^{\ast}}_{11}-H\sum_{k}\bar h^{1^{\ast}}_{1k}\bar h^{1^{\ast}}_{2k}=0.
\end{align*}
Namely,
\begin{align}\label{3.1-19}
\bar h^{1^{\ast}}_{11}-(\bar h^{1^{\ast}}_{11}+\bar h^{1^{\ast}}_{22})\big((\bar h^{1^{\ast}}_{11})^{2}+(\bar h^{2^{\ast}}_{11})^{2}\big)=0, \ \ \bar h^{2^{\ast}}_{11}(1-H^{2})=0.
\end{align}
Since $\bar h^{1^{\ast}}_{11}\bar h^{1^{\ast}}_{22}
-(\bar h^{2^{\ast}}_{11})^{2}\neq 0$, \eqref{3.1-16} and \eqref{3.1-19} imply
\begin{align*}
H^{2}\neq1, \ \ \bar h^{2^{\ast}}_{11}=0, \ \ \bar h^{1^{\ast}}_{11}\neq0, \ \ 1-\bar h^{1^{\ast}}_{11}(\bar h^{1^{\ast}}_{11}+\bar h^{1^{\ast}}_{22})=0, \ \ (\bar h^{1^{\ast}}_{11})^{2}+(\bar h^{1^{\ast}}_{22})^{2}=1.
\end{align*}
Thus,
\begin{align*}
\bar h^{1^{\ast}}_{11}=\bar h^{1^{\ast}}_{22}, \ \ H^{2}=2, \ \  \bar S=2.
\end{align*}

\noindent {\bf Case 2: $\bar h^{1^{\ast}}_{11}\bar h^{1^{\ast}}_{22}
-(\bar h^{2^{\ast}}_{11})^{2}=0$.}

\noindent
Under the condition of Case 2,  it is obvious to draw $\bar h^{1^{\ast}}_{11}+3\bar h^{1^{\ast}}_{22}\neq0$, otherwise we would have

$$0=(\bar h^{1^{\ast}}_{11}+3\bar h^{1^{\ast}}_{22})^{2}=(\bar h^{1^{\ast}}_{11})^{2}+9(\bar h^{1^{\ast}}_{22})^{2}+6(\bar h^{2^{\ast}}_{11})^{2},$$
which implies $\bar h^{p^{\ast}}_{ij}=0$ for any $i, j, p=1, 2$. This contradicts the fact that $\sup S=\sum_{i,j,p}\bar h^{p^{\ast}}_{ij}>0$. In addition, by using \eqref{3.1-3} and \eqref{3.1-14}, we can obtain
\begin{align}\label{3.1-20}
&H^{2}=(\bar h^{1^{\ast}}_{11})^{2}+(\bar h^{1^{\ast}}_{22})^{2}+2(\bar h^{2^{\ast}}_{11})^{2},  \ \ \bar K=\frac{1}{2}(H^{2}-\bar S)=-\big((\bar h^{1^{\ast}}_{22})^{2}+(\bar h^{2^{\ast}}_{11})^{2}\big), \\
&\nonumber\big(\lim_{t\rightarrow\infty}h^{1^{\ast}}_{111}(p_{t})\big)^{2}
=\big(\lim_{t\rightarrow\infty}h^{1^{\ast}}_{221}(p_{t})\big)^{2}
=\frac{4(H^{2}-1)\big(1-\sum_{k}(\bar h^{k^{\ast}}_{11})^{2}\big)^{2}}{H^{2}} , \\
&\nonumber(\bar H^{2^{\ast}}_{,2})^{2}=H^{2}(H^{2}-1).
\end{align}
At this point, we will further divide into the following two subcases, where subcase 2.2 does not exist.

\noindent{\bf Subcase 2.1: $\bar h^{2^{\ast}}_{11}=0$.}

\noindent
Since $\bar h^{2^{\ast}}_{11}=0$, we know that either $\bar h^{1^{\ast}}_{11}=0$ or $\bar h^{1^{\ast}}_{22}=0$.
If $\bar h^{1^{\ast}}_{11}=0$, we get $\bar h^{1^{\ast}}_{22}=H$ and \eqref{3.1-20} implies that
\begin{align}\label{3.1-21}
\big(\lim_{t\rightarrow\infty}h^{1^{\ast}}_{111}(p_{t})\big)^{2}
=\big(\lim_{t\rightarrow\infty}h^{1^{\ast}}_{221}(p_{t})\big)^{2}=\frac{4(H^{2}-1)}{H^{2}}.
\end{align}
From \eqref{3.1-7}, \eqref{3.1-20} and \eqref{3.1-21}, we obtain
\begin{align*}
\lim_{t\rightarrow\infty}h^{1^{\ast}}_{111}(p_{t})=\lim_{t\rightarrow\infty}h^{1^{\ast}}_{221}(p_{t})=\lim_{t\rightarrow\infty}h^{2^{\ast}}_{111}(p_{t})=0, \ \ H^{2}=1, \ \ \bar H^{2^{\ast}}_{,2}=0, \ \ \bar S=3,
\end{align*}
where $|\lim_{t\rightarrow\infty}h^{2^{\ast}}_{222}(p_{t})|<\infty$.

\noindent
Namely,
\begin{align*}
\bar K=\frac{1}{2}(H^{2}-\bar S)=-1, \ \ \lim_{t\rightarrow\infty}h^{p^{\ast}}_{ijk}(p_{t})=0, \ \ \bar H^{p^{\ast}}_{,i}=0, \ \ i, j, k, p=1, 2.
\end{align*}
Thus, \eqref{2.1-16} yields
\begin{align*}
\lim_{t\rightarrow\infty}\sum_{i,j,p}(H^{p^{\ast}}_{,ij})^{2}(p_{t})=2\bar K=-2.
\end{align*}
This is impossible.

\noindent
If $\bar h^{1^{\ast}}_{22}=0$, we have $\bar h^{1^{\ast}}_{11}=H$.
It follows from \eqref{3.1-20} that
\begin{align}\label{3.1-22} \big(\lim_{t\rightarrow\infty}h^{1^{\ast}}_{111}(p_{t})\big)^{2}
=\big(\lim_{t\rightarrow\infty}h^{1^{\ast}}_{221}(p_{t})\big)^{2}=\frac{4(H^{2}-1)^{3}}{H^{2}}.
\end{align}
Since $|\lim_{t\rightarrow\infty}h^{2^{\ast}}_{222}(p_{t})|<\infty$ from \eqref{3.1-14}, by the first equation of \eqref{3.1-7} and \eqref{3.1-22}, we obtain
\begin{align*}
\lim_{t\rightarrow\infty}h^{1^{\ast}}_{111}(p_{t})=0, \ \ H^{2}=1, \ \ \bar S=1.
\end{align*}

\noindent{\bf Subcase 2.2: $\bar h^{2^{\ast}}_{11}\neq0$.}

\noindent
Since $\bar h^{2^{\ast}}_{11}\neq0$, it can be naturally obtained that $$\bar h^{1^{\ast}}_{11}\bar h^{1^{\ast}}_{22}\neq 0, \ \ \bar K<0, \ \ |\lim_{t\rightarrow\infty}h^{2^{\ast}}_{111}(p_{t})|<\infty, \ \ |\lim_{t\rightarrow\infty}h^{1^{\ast}}_{222}(p_{t})|<\infty$$ from the second equation of \eqref{3.1-7}, \eqref{3.1-20} and $|\lim_{t\rightarrow\infty}h^{2^{\ast}}_{222}(p_{t})|<\infty$.

\noindent
Besides, $\bar h^{1^{\ast}}_{11}\bar h^{1^{\ast}}_{22}\neq 0$ implies that there exists some points $p_{t}$ such that $h^{1^{\ast}}_{22}(p_{t})\neq0$.

\noindent
By use of the first equation of \eqref{3.1-11}, the first equation of \eqref{3.1-13} and \eqref{3.1-14}, we have that
\begin{align*}
&\bar H^{2^{\ast}}_{,2}\lim_{t\rightarrow\infty}H^{2^{\ast}}_{,21}(p_{t})
=\frac{1}{2}H^{3}\lim_{t\rightarrow\infty}h^{1^{\ast}}_{111}(p_{t}), \\
&\bar H^{2^{\ast}}_{,2}\lim_{t\rightarrow\infty}H^{2^{\ast}}_{,21}(p_{t})
=H\big(-\lim_{t\rightarrow\infty}h^{1^{\ast}}_{111}(p_{t})
+\lim_{t\rightarrow\infty}h^{2^{\ast}}_{222}(p_{t})\big)
\big(1-(\bar h^{1^{\ast}}_{11})^{2}-(\bar h^{2^{\ast}}_{11})^{2}\big).
\end{align*}
Namely,
\begin{align}\label{3.1-23}
\big((\bar h^{1^{\ast}}_{11})^{2}-(\bar h^{1^{\ast}}_{22})^{2}-2\big)\lim_{t\rightarrow\infty}h^{1^{\ast}}_{111}(p_{t})-
2\big((\bar h^{1^{\ast}}_{11})^{2}+(\bar h^{2^{\ast}}_{11})^{2}-1\big)\lim_{t\rightarrow\infty}h^{2^{\ast}}_{222}(p_{t})=0,
\end{align}
at this point, we need to use $H^{2}=(\bar h^{1^{\ast}}_{11})^{2}+(\bar h^{1^{\ast}}_{22})^{2}+2(\bar h^{2^{\ast}}_{11})^{2}$ to obtain the above equation.

\noindent
Besides, it follows from \eqref{3.1-2}, \eqref{3.1-4} and \eqref{3.1-5} that
\begin{align}\label{3.1-24}
H^{2^{\ast}}_{,2}=H^{1^{\ast}}_{,1}+H^{2^{\ast}}_{,2}
=H\langle x, e_{1} \rangle, \ \
\langle x, e_{2} \rangle=-\frac{h^{2^{\ast}}_{11}}{Hh^{1^{\ast}}_{22}}H^{2^{\ast}}_{,2},
\end{align}
which implies
\begin{align*}
|\lim_{t\rightarrow\infty}\langle x, e_{1} \rangle(p_{t})|<\infty,  \ \ |\lim_{t\rightarrow\infty}\langle x, e_{2} \rangle(p_{t})|<\infty.
\end{align*}
Taking the limit in \eqref{3.1-7} and using $\bar h^{1^{\ast}}_{11}\bar h^{1^{\ast}}_{22}
-(\bar h^{2^{\ast}}_{11})^{2}=0$, it can be concluded that
\begin{align}\label{3.1-25}
&\big((\bar h^{1^{\ast}}_{11})^{2}+9(\bar h^{1^{\ast}}_{22})^{2}+6(\bar h^{2^{\ast}}_{11})^{2}\big)\lim_{t\rightarrow\infty}h^{1^{\ast}}_{111}(p_{t})
+4(\bar h^{2^{\ast}}_{11})^{2}\lim_{t\rightarrow\infty}h^{2^{\ast}}_{222}(p_{t})=0, \\
&\nonumber\big((\bar h^{1^{\ast}}_{11})^{2}+9(\bar h^{1^{\ast}}_{22})^{2}+6(\bar h^{2^{\ast}}_{11})^{2}\big) \lim_{t\rightarrow\infty}h^{2^{\ast}}_{111}(p_{t})
-h^{\bar 2^{\ast}}_{11}(\bar h^{1^{\ast}}_{11}-3\bar h^{1^{\ast}}_{22})
\lim_{t\rightarrow\infty}h^{2^{\ast}}_{222}(p_{t})
=0.
\end{align}

Next, we will discuss the following situations: $H^{2}\neq1$ and $H^{2}=1$.
In fact, under the fact of $\bar h^{2^{\ast}}_{11}\neq0$, both of these situations do not exist.

If $H^{2}\neq1$, \eqref{3.1-20} and \eqref{3.1-25} yield that
\begin{align}\label{3.1-26}
\bar H^{2^{\ast}}_{,2}\neq0, \ \ \lim_{t\rightarrow\infty}h^{1^{\ast}}_{111}(p_{t})\neq0, \ \
\lim_{t\rightarrow\infty}h^{1^{\ast}}_{221}(p_{t})\neq0.
\end{align}
It follows from \eqref{3.1-23} and the first equation of \eqref{3.1-25} that
\begin{align*}
&\big((\bar h^{1^{\ast}}_{11})^{2}+9(\bar h^{1^{\ast}}_{22})^{2}+6(\bar h^{2^{\ast}}_{11})^{2}\big)\big((\bar h^{1^{\ast}}_{11})^{2}+(\bar h^{2^{\ast}}_{11})^{2}-1\big)
+2(\bar h^{2^{\ast}}_{11})^{2}\big((\bar h^{1^{\ast}}_{11})^{2}-(\bar h^{1^{\ast}}_{22})^{2}-2\big)=0
\end{align*}
since $\lim_{t\rightarrow\infty}h^{1^{\ast}}_{111}(p_{t})\neq0$ from \eqref{3.1-26}.

\noindent
Expanding the above polynomials yields
\begin{align}\label{3.1-27}
&-(\bar h^{1^{\ast}}_{11})^{2}-9(\bar h^{1^{\ast}}_{22})^{2}-10(\bar h^{2^{\ast}}_{11})^{2}+(\bar h^{1^{\ast}}_{11})^{4}
+9(\bar h^{1^{\ast}}_{11})^{2}(\bar h^{1^{\ast}}_{22})^{2}
+9(\bar h^{1^{\ast}}_{11})^{2}(\bar h^{2^{\ast}}_{11})^{2} \\
&\nonumber+7(\bar h^{1^{\ast}}_{22})^{2}(\bar h^{2^{\ast}}_{11})^{2}+6(\bar h^{2^{\ast}}_{11})^{4}=0.
\end{align}
By use of the second equation of \eqref{3.1-11}, the second equation of \eqref{3.1-13}, the fourth equation of \eqref{3.1-20} and \eqref{3.1-24}, we infer
\begin{align*}
\bar H^{2^{\ast}}_{,2}\lim_{t\rightarrow\infty}H^{2^{\ast}}_{,22}(p_{t})
=&H^{2}\big(2\bar h^{1^{\ast}}_{22}\lim_{t\rightarrow\infty}h^{2^{\ast}}_{111}(p_{t})+2\bar h^{2^{\ast}}_{11}\lim_{t\rightarrow\infty}h^{2^{\ast}}_{222}(p_{t})\big), \\
\bar H^{2^{\ast}}_{,2}\lim_{t\rightarrow\infty}H^{2^{\ast}}_{,22}(p_{t})=
&\bar H^{2^{\ast}}_{,2}\big(\bar H^{2^{\ast}}_{,2}\lim_{t\rightarrow\infty}\langle x, e_{2}\rangle(p_{t})-H^{2}\bar h^{2^{\ast}}_{11}\big) \\
=&-\frac{H}{\bar h^{1^{\ast}}_{22}}\big((H^{2}-1)\bar h^{2^{\ast}}_{11}+H\bar h^{1^{\ast}}_{22}\bar h^{2^{\ast}}_{11}\big)\cdot\bar H^{2^{\ast}}_{,2}.
\end{align*}
Thus, from $$\bar H^{2^{\ast}}_{,2}=-\lim_{t\rightarrow\infty}h^{1^{\ast}}_{111}(p_{t})
+\lim_{t\rightarrow\infty}h^{2^{\ast}}_{222}(p_{t}),$$ we get
\begin{align}\label{3.1-28}
&\big((H^{2}-1)\bar h^{2^{\ast}}_{11}+H\bar h^{1^{\ast}}_{22}\bar h^{2^{\ast}}_{11}\big)
\lim_{t\rightarrow\infty}h^{1^{\ast}}_{111}(p_{t})-2H(\bar h^{1^{\ast}}_{22})^{2}\lim_{t\rightarrow\infty}h^{2^{\ast}}_{111}(p_{t}) \\
&\nonumber-\big((H^{2}-1)\bar h^{2^{\ast}}_{11}+3H\bar h^{1^{\ast}}_{22}\bar h^{2^{\ast}}_{11}\big)
\lim_{t\rightarrow\infty}h^{2^{\ast}}_{222}(p_{t})=0.
\end{align}
Combining the second equation of \eqref{3.1-25} and \eqref{3.1-28}, we have
\begin{align}\label{3.1-29}
&\big((\bar h^{1^{\ast}}_{11})^{2}+9(\bar h^{1^{\ast}}_{22})^{2}+6(\bar h^{2^{\ast}}_{11})^{2}\big)\big((H^{2}-1)\bar h^{2^{\ast}}_{11}+H\bar h^{1^{\ast}}_{22}\bar h^{2^{\ast}}_{11}\big)
\lim_{t\rightarrow\infty}h^{1^{\ast}}_{111}(p_{t}) \\
&\nonumber-\Big(\big((\bar h^{1^{\ast}}_{11})^{2}+9(\bar h^{1^{\ast}}_{22})^{2}+6(\bar h^{2^{\ast}}_{11})^{2}\big)\big((H^{2}-1)\bar h^{2^{\ast}}_{11}+3H\bar h^{1^{\ast}}_{22}\bar h^{2^{\ast}}_{11}\big) \\
&\nonumber+2H(\bar h^{1^{\ast}}_{22})^{2}h^{\bar 2^{\ast}}_{11}(\bar h^{1^{\ast}}_{11}-3\bar h^{1^{\ast}}_{22})\Big)
\lim_{t\rightarrow\infty}h^{2^{\ast}}_{222}(p_{t})=0.
\end{align}
Then the first equation of \eqref{3.1-25} and \eqref{3.1-29} yield that
\begin{align*}
&\big((\bar h^{1^{\ast}}_{11})^{2}+9(\bar h^{1^{\ast}}_{22})^{2}+6(\bar h^{2^{\ast}}_{11})^{2}\big)\big((H^{2}-1)\bar h^{2^{\ast}}_{11}+3H\bar h^{1^{\ast}}_{22}\bar h^{2^{\ast}}_{11}\big) \\
&+2H(\bar h^{1^{\ast}}_{22})^{2}h^{\bar 2^{\ast}}_{11}(\bar h^{1^{\ast}}_{11}-3\bar h^{1^{\ast}}_{22})+4(h^{\bar 2^{\ast}}_{11})^{2}\big((H^{2}-1)\bar h^{2^{\ast}}_{11}+H\bar h^{1^{\ast}}_{22}\bar h^{2^{\ast}}_{11}\big)=0.
\end{align*}
since $\lim_{t\rightarrow\infty}h^{1^{\ast}}_{111}(p_{t})\neq0$ from \eqref{3.1-26}.

\noindent
By direct calculation, it is obtained that
\begin{align}\label{3.1-30}
&-(\bar h^{1^{\ast}}_{11})^{2}-9(\bar h^{1^{\ast}}_{22})^{2}-10(\bar h^{2^{\ast}}_{11})^{2}+(\bar h^{1^{\ast}}_{11})^{4}
+3(\bar h^{1^{\ast}}_{11})^{3}\bar h^{1^{\ast}}_{22}
+15(\bar h^{1^{\ast}}_{11})^{2}(\bar h^{1^{\ast}}_{22})^{2}
\\
&\nonumber+23\bar h^{1^{\ast}}_{11}(\bar h^{1^{\ast}}_{22})^{3}+30(\bar h^{1^{\ast}}_{22})^{4}
+12(\bar h^{1^{\ast}}_{11})^{2}(\bar h^{2^{\ast}}_{11})^{2}+22\bar h^{1^{\ast}}_{11}\bar h^{1^{\ast}}_{22}(\bar h^{2^{\ast}}_{11})^{2} \\
&\nonumber+50(\bar h^{1^{\ast}}_{22})^{2}(\bar h^{2^{\ast}}_{11})^{2}+20(\bar h^{2^{\ast}}_{11})^{4}=0,
\end{align}
where $H=\bar h^{1^{\ast}}_{11}+\bar h^{1^{\ast}}_{22}$ and $H^{2}=(\bar h^{1^{\ast}}_{11})^{2}+(\bar h^{1^{\ast}}_{22})^{2}+2(\bar h^{2^{\ast}}_{11})^{2}$.

\noindent
Subtracting \eqref{3.1-27} and \eqref{3.1-30} yields
\begin{align*}
&3(\bar h^{1^{\ast}}_{11})^{3}\bar h^{1^{\ast}}_{22}
+6(\bar h^{1^{\ast}}_{11})^{2}(\bar h^{1^{\ast}}_{22})^{2}+23\bar h^{1^{\ast}}_{11}(\bar h^{1^{\ast}}_{22})^{3}+30(\bar h^{1^{\ast}}_{22})^{4}
+3(\bar h^{1^{\ast}}_{11})^{2}(\bar h^{2^{\ast}}_{11})^{2}
\\
&+22\bar h^{1^{\ast}}_{11}\bar h^{1^{\ast}}_{22}(\bar h^{2^{\ast}}_{11})^{2}
+43(\bar h^{1^{\ast}}_{22})^{2}(\bar h^{2^{\ast}}_{11})^{2}+14(\bar h^{2^{\ast}}_{11})^{4}=0.
\end{align*}
Thus, by $\bar h^{1^{\ast}}_{11}\bar h^{1^{\ast}}_{22}
-(\bar h^{2^{\ast}}_{11})^{2}=0$, we know
\begin{align*}
(\bar h^{1^{\ast}}_{11})^{2}(\bar h^{2^{\ast}}_{11})^{2}
+11(\bar h^{1^{\ast}}_{22})^{2}(\bar h^{2^{\ast}}_{11})^{2}+5(\bar h^{1^{\ast}}_{22})^{4}
+7(\bar h^{2^{\ast}}_{11})^{4}=0,
\end{align*}
which implies
\begin{align*}
\bar h^{1^{\ast}}_{22}=\bar h^{2^{\ast}}_{11}=0.
\end{align*}
This contradicts the fact that $\bar h^{2^{\ast}}_{11}\neq0$.

If $H^{2}=1$, it is natural to draw the following conclusion from \eqref{3.1-20},
\begin{align*}
 \lim_{t\rightarrow\infty}h^{1^{\ast}}_{111}(p_{t})
=\lim_{t\rightarrow\infty}h^{1^{\ast}}_{221}(p_{t})=0, \ \ \bar H^{2^{\ast}}_{,2}=0, \ \
\lim_{t\rightarrow\infty}h^{2^{\ast}}_{222}(p_{t})=0.
\end{align*}
And by $\bar h^{2^{\ast}}_{11}\neq0$ and $\lim_{t\rightarrow\infty}h^{1^{\ast}}_{111}(p_{t})=0$, the first equation of \eqref{3.1-10} yields
\begin{align*}
\lim_{t\rightarrow\infty}h^{2^{\ast}}_{111}(p_{t})=0.
\end{align*}
Based on the above discussion, it can be concluded that
\begin{align}\label{3.1-31}
\lim_{t\rightarrow\infty}h^{p^{\ast}}_{ijk}(p_{t})=0, \ \ \bar H^{p^{\ast}}_{,i}=0, \ \ i, j, k, p=1,2.
\end{align}
Since $\bar h^{1^{\ast}}_{11}\bar h^{1^{\ast}}_{22}
-(\bar h^{2^{\ast}}_{11})^{2}=0$, $H^{2}=(\bar h^{1^{\ast}}_{11})^{2}+(\bar h^{1^{\ast}}_{22})^{2}+2(\bar h^{2^{\ast}}_{11})^{2}=1$ and $\bar K=-\big((\bar h^{1^{\ast}}_{22})^{2}+(\bar h^{2^{\ast}}_{11})^{2}\big)$, by a simple calculations show that
\begin{align*}
&-\bar KH^{3}\bar h^{1^{\ast}}_{11}=-\bar KH\bar h^{1^{\ast}}_{11}=-\bar K\big((\bar h^{1^{\ast}}_{11})^{2}+(\bar h^{2^{\ast}}_{11})^{2}\big)=-\bar K(\bar K+1),
\end{align*}
\begin{align*}
-H^{2}\sum_{q}\big(\sum_{i,j}\bar h^{1^{\ast}}_{ij}\bar h^{q^{\ast}}_{ij}\big)^{2}=
-\Big(\big(\sum_{i,j}(\bar h^{1^{\ast}}_{ij})^{2}\big)^{2}+\big(\sum_{i,j}\bar h^{1^{\ast}}_{ij}\bar h^{2^{\ast}}_{ij}\big)^{2}\Big)=-\big(1+(\bar h^{2^{\ast}}_{11})^{2}\big)
\end{align*}
and
\begin{align*}
-H^{3}\sum_{j,k,l}\bar h^{1^{\ast}}_{jk}\bar h^{1^{\ast}}_{jl}\bar h^{1^{\ast}}_{kl}
=&-H\big(\sum_{l}\bar h^{1^{\ast}}_{11}(\bar h^{1^{\ast}}_{1l})^{2}+\sum_{l}\bar h^{1^{\ast}}_{22}(\bar h^{1^{\ast}}_{2l})^{2}+2\sum_{l}\bar h^{1^{\ast}}_{12}\bar h^{1^{\ast}}_{1l}\bar h^{1^{\ast}}_{2l}\big) \\
=&-H\Big(
(\bar h^{1^{\ast}}_{11})^{3}+3\bar h^{1^{\ast}}_{11}(\bar h^{2^{\ast}}_{11})^{2}+(\bar h^{1^{\ast}}_{22})^{3}+3\bar h^{1^{\ast}}_{22}(\bar h^{2^{\ast}}_{11})^{2}\Big) \\
=&-H\Big(H\big((\bar h^{1^{\ast}}_{11})^{2}+(\bar h^{1^{\ast}}_{22})^{2}-\bar h^{1^{\ast}}_{11}\bar h^{1^{\ast}}_{22}\big)+3H(\bar h^{2^{\ast}}_{11})^{2}\Big) \\
=&-\big((\bar h^{1^{\ast}}_{11})^{2}+(\bar h^{1^{\ast}}_{22})^{2}+2(\bar h^{2^{\ast}}_{11})^{2}\big)= -1.
\end{align*}
Thus, combining \eqref{2.1-16} and \eqref{3.1-31}, we infer
\begin{align}\label{3.1-32}
\lim_{t\rightarrow\infty}\sum_{i,j,p}(H^{p^{\ast}}_{,ij})^{2}(p_{t})
=&(3\bar K+2)H^{2}\sum_{j,k}(\bar h^{1^{\ast}}_{jk})^{2}
-\bar K\big(H^{4}+H^{3}\bar h^{1^{\ast}}_{11} \big) \\
&-H^{2}\sum_{q}\big(\sum_{i,j}\bar h^{1^{\ast}}_{ij}\bar h^{q^{\ast}}_{ij}\big)^{2}
-H^{3}\sum_{j,k,l}\bar h^{1^{\ast}}_{jk}\bar h^{1^{\ast}}_{jl}\bar h^{1^{\ast}}_{kl} \nonumber \\
=&-(\bar K)^{2}+\bar K-(\bar h^{2^{\ast}}_{11})^{2}<0,\nonumber
\end{align}
where $\bar K=-\big((\bar h^{1^{\ast}}_{22})^{2}+(\bar h^{2^{\ast}}_{11})^{2}\big)<0$. This is impossible.
The proof of the Proposition \ref{proposition 3.1} is finished.
\end{proof}

Using the same proof of the Proposition \ref{proposition 3.1} and applying the generalized
maximum principle to the function $-S$, we can derive the following proposition.

\begin{proposition}\label{proposition 3.2}
Let $x:M^{2}\rightarrow \mathbb{C}^{2}$
be a Lagrangian shrinker with non-zero constant squared norm $|\vec{H}|^{2}$ of the mean curvature vector. If the squared norm $S$ of the second fundamental form is bounded from above, then $H^{2}=2$ and $\inf S=2$, or $H^{2}=1$ and $\inf S=1$.
\end{proposition}

\vskip3mm
\noindent
{\it Proof of Theorem \ref{theorem 1.1}}.
If $|\vec{H}|^{2}\equiv0$, we know $S=0$. In fact, it follows from the definition of lagrangian self-shrinker and the first equation of \eqref{2.1-11} that
$$H^{p^{\ast}}=-\langle x, e_{p^{\ast}}\rangle=0, \ \ \sum_{k}h^{p^{\ast}}_{ik}\langle x, e_{k}\rangle=0, \ \ i, p=1, 2.$$
Namely,
$$h^{1^{\ast}}_{11}+h^{1^{\ast}}_{22}=0, \ \ h^{2^{\ast}}_{11}+h^{2^{\ast}}_{22}=0$$
and
\begin{equation*}
h^{1^{\ast}}_{11}\langle x, e_{1}\rangle+h^{1^{\ast}}_{12}\langle x, e_{2}\rangle=0, \ \
h^{1^{\ast}}_{21}\langle x, e_{1}\rangle+h^{1^{\ast}}_{22}\langle x, e_{2}\rangle=0.
\end{equation*}
Then by the symmetry of indices, we infer
$$\big((h^{1^{\ast}}_{11})^{2}+(h^{2^{\ast}}_{11})^{2}\big)\langle x, e_{k}\rangle=0, \ \ k=1,2.$$
Assume $(h^{1^{\ast}}_{11})^{2}+(h^{2^{\ast}}_{11})^{2}\neq0$, we obtain
$$0=\langle x, e_{k}\rangle_{,k}=1+\sum_{p}h^{p^{\ast}}_{kk}\langle x, e_{p^{\ast}}\rangle=1.$$
It is impossible.
If $|\vec{H}|^{2}$ is non-zero constant,
from the Proposition \ref{proposition 3.1} and the Proposition \ref{proposition 3.2}, we know
$H^{2}=2$ and $S=2$, or $H^{2}=1$ and $S=1$. So we can use the
classification theorem of Cheng and Wei \cite{CHW} to complete the proof of
the Theorem \ref{theorem 1.1}.
\begin{flushright}
$\square$
\end{flushright}

\noindent{\bf Acknowledgements.}
The first author was partially supported by NSFC Grant No.12401060, Natural Science Foundation of Henan Province Grant No.242300421686. The second author was partially supported by Natural Science Foundation of Henan Province Grant No.242300420641, Henan Provincial Postdoctoral Research Project Grant No.HN2022156. The third author was partly supported by NSFC Grant No.12171164.

\end{document}